\documentclass[11pt]{article}
\usepackage{amsmath,amsfonts,amssymb,amsthm,fullpage}
\input xy
\xyoption{all}

\newcommand{\RR}{\mathbb R}
\newcommand{\ZZ}{\mathbb Z}
\newcommand{\QQ}{\mathbb Q}
\newcommand{\CC}{\mathbb C}

\newcommand{\C}{\mathcal C}

\def\TT{\mathbb{T}}
\def\A{\mathcal{A}}

\newcommand{\sheaf}{\mathcal{O}}
\newcommand{\map}{\longrightarrow}

\newcommand{\beq}{\begin{eqnarray}}
\newcommand{\beqn}{\begin{eqnarray*}}
\newcommand{\eeq}{\end{eqnarray}}
\newcommand{\eeqn}{\end{eqnarray*}}

\newtheorem{thm}{Theorem}

\newtheorem{lem}[thm]{Lemma}
\newtheorem{prop}[thm]{Proposition}
\newtheorem{cor}[thm]{Corollary}
\newtheorem{ex}{Example}
\newtheorem{defn}[thm]{Definition}
\theoremstyle{remark}
\newtheorem{rem}[thm]{Remark}
\newtheorem*{theorem*}{Theorem}

\def\d{\mathtt{\delta}}
\def\A{\mathcal{A}}
\def\FrVec{{\sf FrVect}}

\def\Ann{\textup{Ann}}
\def\BB{\mathcal B}
\def\Coh{\textup{Coh}}
\def\coker{\textup{coker}}
\def\deg{\textup{deg}}

\def\F{{\sf F}}
\def\tan{\textup{tan}}
\def\arctan{\textup{arctan}}
\def\K{\textup{K}}

\def\I{\textup{I}}

\def\mod{\textup{mod}}
\def\ssz{\sf SS_{\mathcal Z}}
\def\End{\textup{End}}
\def\Pic{\textup{Pic}}

\def\Ext{\textup{Ext}}
\def\Hom{\textup{Hom}}

\def\Mod{{\sf Mod}}
\def\op{\textup{op}}
\def\NF{{\sf NF}}
\def\Rep{{\sf Rep}}

\def\rk{\textup{rk}}
\def\rwd{\textup{wd}}
\def\Sheaf{\mathcal}
\def\SS{\mathbb{S}}
\def\cS{\mathcal{S}}
\def\Vec{{\sf Vect}}
\def\trans{{\bf{T}}}

\title{Noncommutative tori and the Riemann--Hilbert correspondence}
\author{Snigdhayan Mahanta$^1$ and Walter D. van Suijlekom$^2$\\[4mm]
$^1$ Department of Mathematics, University of Toronto, \\
40 St George Street, Room 6290,
Toronto, ON M5S 2E4, Canada\\[3mm]
$^2$Institute for Mathematics, Astrophysics and Particle Physics\\
Faculty of Science, Radboud Universiteit Nijmegen\\
Toernooiveld 1,
6525 ED Nijmegen, The Netherlands
}

\date{\today}

\begin{document}

\maketitle

\begin{abstract}
We study the interplay between noncommutative tori and noncommutative elliptic curves through a category of equivariant differential modules on
$\CC^*$.  We functorially relate this category to the category of
holomorphic vector bundles on noncommutative tori as introduced by Polishchuk
and Schwarz and study the induced map between the corresponding K-theories. In addition, there is a forgetful functor to the category of noncommutative elliptic curves of Soibelman and Vologodsky, as well as the forgetful functor to the category of vector bundles on $\CC^*$ with regular singular connections. 

The category that we consider has the nice property of being a Tannakian category, hence it is equivalent to the category of representations of an affine group scheme.  Via an equivariant version of the Riemann--Hilbert correspondence we determine this group scheme to be (the algebraic hull of) $\ZZ^2$. We also obtain a full subcategory of the category of holomorphic bundles on the noncommutative torus, which is equivalent to the category of representations of $\ZZ$. This group is the proposed topological fundamental group of the noncommutative torus (understood as a degenerate elliptic curve) and we study Nori's notion of \'etale fundamental group in this context.

\end{abstract}

\begin{center}
{\Large {\bf Introduction}}
\end{center}

Noncommutative geometry in its various forms has come to the forefront of
mathematical research lately and noncommutative tori constitute perhaps the
most extensively studied class of examples of noncommutative differentiable manifolds. They were introduced by
Connes during the early eighties \cite{Con80} and were
systematically studied by Connes \cite{Con80}, Rieffel \cite{Rie,Rie2} and others. Recently
Polishchuk and Schwarz have provided a new perspective on them which is
quite amenable to techniques in algebraic geometry \cite{PolSch, Pol2}. At the same time Soibelman and
Vologodsky have introduced noncommutative elliptic curves
as certain equivariant categories of coherent sheaves \cite{SoiVol}. The guiding principle
behind both constructions is replacing a mathematical object by its category
of appropriately defined representations, {\it viz.,} vector bundles
with connections in the former case, denoted by $\Vec(\TT_\theta^\tau)$, and
coherent sheaves in the latter, denoted by $\BB_q$, where $q=e^{2 \pi i \theta}$ and $\theta$ is an irrational number.  

In this article we try to connect the above two constructions by introducing an intermediate category $\BB_q^\tau$. 
Besides the existence of a forgetful functor from $\BB_q^\tau$ to $\BB_q$ (as
the notation might suggest), we construct a faithful and exact functor from
$\BB_q^\tau$ to $\Vec(\TT_\theta^\tau)$. One may view $\BB_q^\tau$ as a sort of categorical `correspondence' between $\Vec(\TT_\theta^\tau)$ and $\BB_q$. 
The category $\BB_q^\tau$ turns out to be a Tannakian category. Via an equivariant version of the Riemann--Hilbert
correspondence we show that it is equivalent to the category of finite
dimensional representations of (the algebraic hull of) $\ZZ^2$ (see Theorem
\ref{thm:RH}). This allows us to describe the K-theory of
$\BB_q^\tau$ as the free abelian group generated by two copies of $\CC^*$ (see Corollary \ref{ktheory}). 

The final part of the paper is a little speculative in nature, in which we propose a fundamental group for the noncommutative torus, both topological and \'etale. For the latter, Nori's approach \cite{Nor1,Nor2} to \'etale fundamental groups of smooth quasiprojective curves involving Tannakian categories plays a central motivating role.

This paper is organized as follows.
In the first section we briefly review the main results of \cite{PolSch}, including the basic definitions and examples. We also discuss the rudiments of noncommutative tori, which are relevant for
our purposes as it is known that there are several ways of looking at them. We
also show that there is a certain {\it modularity} property satisfied by the
categories $\Vec(\TT_\theta^\tau)$ (see Proposition \ref{modularity}).

In the second section we first provide a motivation for the definition of the
categories $\BB_q^\tau$ and then construct a faithful and exact functor from
$\BB_q^\tau$ into $\Vec(\TT_\theta^\tau)$. We also give a description of the
image of our functor and discuss the induced map on the K-theories of the
corresponding categories. 

In the third section we start by briefly recalling some preliminaries of
Tannakian categories. 
We explain the structure of a Tannakian
category on the category $\BB_q^\tau$ and prove an equivariant version of
the Riemann--Hilbert correspondence on $\CC^*$. Via this correspondence, we
find that $\BB_q^\tau$ is equivalent to the category of finite dimensional representations of $\ZZ^2$. As a consequence we are able to compute the
K-theory of $\BB_q^\tau$.

We conclude with the proposal for the fundamental group of the noncommutative torus as alluded to before. After a short discussion on Nori finite bundles involving stability conditions, we establish a full subcategory $\BB^\tau$ of $\Vec(\TT_\theta^\tau)$ which is a Tannakian category, and equivalent to $\Rep(\ZZ)$. We define a Tannakian subcategory of $\BB^\tau$, whose associated group scheme is our proposal for the \'etale fundamental group of the noncommutative torus. 
 

\medskip

\noindent
{\bf Convention.} In this article, unless otherwise stated, $\theta$ is always
assumed to be irrational and $\tau$ in the lower half plane as in
\cite{PolSch}. The ground field is also assumed to be $\CC$.

\section*{Acknowledgements}Both authors would like to thank Matilde Marcolli
for intensive discussions, as well as Masoud Khalkhali, Emanuele Macr\'i and Jorge Plazas for helpful
comments. The first author would also like to thank Vasilisa Shramchenko for
indicating the reference \cite{Was} and Christian Kaiser, E. Lee Lady and
Fernando Muro for several useful discussions. The second author would like to
thank Eli Hawkins, Michael M\"uger and Marius van der Put for helpful
comments. Both authors would also like to thank Max-Planck-Institut f\"ur
Mathematik, Bonn for hospitality and support. The first author also acknowledges the hospitality of the Fields Institute, Toronto, and the second author the Hausdorff Research Institute for Mathematics, Bonn, where a part of this work was carried out.

\section{Preliminaries}

We recall some basic facts about the structure of the category of holomorphic bundles over noncommutative tori.

\subsection{Holomorphic bundles on noncommutative tori} \label{HolCat}

The noncommutative torus is a particular case of a transformation group
$C^*$-algebra, with $\ZZ$ acting continuously on the $C^*$-algebra $C(\SS^1)$
of continuous functions on the circle. Pimsner and Voiculescu \cite{PV1} and
separately Rieffel \cite{Rie} studied their K-theory, while Connes analysed their differential structure \cite{Con80}. We will work with the smooth noncommutative torus, which is a dense Fr\'echet subalgebra of this transformation group $C^*$-algebra. 

Let $\theta$ be an irrational real number. 
The algebra of smooth functions $\A_\theta$ on the noncommutative torus $\TT_\theta$ consists of elements of the form $\sum_{(n_1,n_2)\in\ZZ^2}a_{n_1,n_2}U_1^{n_1}U_2^{n_2}$ with $(n_1,n_2)\map a_{n_1,n_2}$ rapidly decreasing and $U_1,U_2$ are unitaries satisfying the commutation relation
\begin{equation}
\label{ComReln}
U_2 U_1 = \exp(2\pi i\theta)U_1 U_2
\end{equation}
A less {\it ad hoc} definition of $\A_\theta$ is given as the {\it smooth crossed product} $C^\infty(\SS^1) \rtimes_\theta \ZZ$ with $\ZZ$ acting on $\SS^1$ by rotation over $\theta$; this is the smooth analogue of the aforementioned transformation group $C^*$-algebra. 
The Fourier transform then establishes the isomorphism $\A_\theta \simeq C^\infty(\SS^1) \rtimes_\theta \ZZ$.

\medskip

\noindent
The two basic derivations $\delta_1$ and $\delta_2$ acting on $\A_\theta$ are as follows,
\begin{align*} 
\delta_j \left( \sum_{(n_1,n_2)\in\ZZ^2}a_{n_1,n_2}U_1^{n_1}U_2^{n_2} \right) 
= 2\pi i \sum_{(n_1,n_2)\in\ZZ^2} n_j a_{n_1,n_2}U_1^{n_1}U_2^{n_2}; \qquad (j=1,2).
\end{align*}
Equivalently, one can define $\delta_1$ and $\delta_2$ by $\delta_j(U_i) =
2\pi i \delta_{ij} U_i$ which is then extended to the whole of $\A_\theta$ by
applying the Leibniz rule. 

The derivations $\delta_1$ and $\delta_2$ are the infinitesimal generators of
the action of a commutative torus $\TT^2$ on $\A_\theta$ by
automorphisms. Inside the complexified Lie algebra generated by $\delta_1$ and
$\delta_2$, we are interested in the vector parametrized by two complex numbers $\omega_1$ and $\omega_2$. We denote 
\begin{equation*}
\delta_\omega = \omega_1 \delta_1 +  \omega_2 \delta_2.
\end{equation*}
If $\omega=(\tau,1)$ we also set $\delta_\tau=\delta_\omega$, which is the
so-called {\it complex structure} on $\A_\theta$ as in \cite{CR87}.

\subsubsection{The category of holomorphic bundles on $\TT_\theta$}
The Serre-Swan Theorem establishes an equivalence between the category of vector bundles over a
topological space $M$ and finitely generated projective modules (henceforth,
for brevity, referred to as finite projective modules) over $C(M)$. 
In this spirit, it makes sense to {\it define} vector bundles over
the noncommutative torus $\TT_\theta$ as finite projective
right $\A_\theta$-modules. 

In \cite{Con80}, Connes has constructed finite projective modules over
$\A_\theta$ that are labelled by a tuple $(m,n)\in\ZZ^2$. Later, in \cite{Rie}
Rieffel has shown that this set, in fact, exhausts the complete set of finite projective modules over $\A_\theta$ (up to isomorphism).

We generalize the category considered by Polishchuk and Schwarz slightly by defining the objects of the category
$\Vec(\TT^\omega_\theta)$ to be finite projective right $\A_\theta$-modules carrying a {\it holomorphic
  structure} which is a lifting of $\delta_\omega$. More
precisely, a holomorphic structure on a finite projective $\A_\theta$-module $E$ is given by a
$\CC$-linear connection $\nabla: E\map E$ satisfying the Leibniz rule,
\begin{equation}
\label{leibniz}
\nabla (ea) = \nabla (e)a +
e\delta_\omega (a); \qquad (\forall e\in E, a\in \A_\theta).
\end{equation}
A morphism $h:E\map E'$ is said to be {\it holomorphic} if it commutes with the
connection, {\it i.e.,} $\nabla_{E'} (he)=h\nabla_E (e)$. These are the
morphisms of the category. 

One defines the cohomology groups $ H^0$ (resp. $H^1$) of $\A_\theta$ with
respect to a holomorphic bundle $E$, equipped with a connection $\nabla$, as the kernel (resp. cokernel) of $\nabla$.

If $\omega=(\tau,1)$, then $\Vec(\TT^\omega_\theta)$ reduces to the category of holomorphic bundles $\Vec(\TT^\tau_\theta)$ as introduced in \cite{PolSch}.

\begin{prop} \label{modularity}
\begin{itemize}
\item[(a)] If $g$ is an element in $SL(2,\ZZ)$, then $\Vec(\TT_\theta^{g\omega}) \simeq \Vec(\TT_\theta^\omega)$.
\item[(b)] If $\omega_2 \neq 0$ and $\tau=\tfrac{\omega_1}{\omega_2}$, then $\Vec(\TT_\theta^\omega) \simeq \Vec(\TT_\theta^{\tau})$.
\end{itemize}
\end{prop}
\begin{proof}
{\it (a)} Given a $g\in SL(2,\ZZ)$, we construct a $\ast$-automorphism
  $\sigma$ of $\A_\theta$ such that $\sigma^{-1}\delta_\omega\sigma =
  \delta_{g\omega}$. Evidently, it is enough to do this for the generators of
  $SL(2,\ZZ)$, {\it i.e.,} $g_1 = \left(\begin{smallmatrix}
                                    1 & 1 \\
                                    0 & 1 \\
                                  \end{smallmatrix}\right)$ and $g_2 =
  \left(\begin{smallmatrix} 
             0 & -1 \\
             1 & 0 \\
       \end{smallmatrix}\right)$. For $g_1$, $\delta_{g_1\omega} = (\omega_1
             +\omega_2)\delta_1 + \delta_2$. We define
             $\sigma_1:\A_\theta\map\A_\theta$ as $\sigma_1(U_1)=U_1U_2$,
             $\sigma_1(U_2) = U_2$. One may easily check that $\sigma_1(U_1)$
             and $\sigma_1(U_2)$ satisfy the commutation relation of
             $\A_\theta$ as in Eqn. \eqref{ComReln} and also that

\beqn
\sigma_1^{-1}\delta_\omega\sigma_1(U_1) = \delta_{g_1\omega}(U_1).
\eeqn 
Similarly, for $U_2$ one may check that the actions of $\delta_{\omega}$
and $\delta_{g_1\omega}$ agree. For $g_2$, $\delta_{g_2\omega} =
-\omega_2\delta_1 + \omega_1\delta_2$ and we define $\sigma_2(U_1) = U_2^{-1}$,
$\sigma_2(U_2) = U_1$. Once again one can easily check that the new generators
satisfy Eqn. \eqref{ComReln} and that the actions of $\delta_{\omega}$ and
$\delta_{g_2\omega}$ agree on $U_1$ and $U_2$. Explicitly, the functor sends
$(\A_\theta,\delta_\omega)$ to $(\A_\theta,\delta_{g_i\omega})$, $i=1,2$, and
twists the module structure by $\sigma_i$, $i=1,2$, {\it i.e.,} $e \cdot a:=
e \sigma_i(a)$, $i=1,2$ and $e \in E$. One verifies that $\nabla$ on $E$ is compatible with
$\delta_{g_i\omega}$, $i=1,2$, with respect to the twisted module
structure. Indeed, 
\beqn
\nabla(e\cdot a) = \nabla(e\sigma_i(a))=\nabla(e) \cdot a + e \cdot \delta_{g_i\omega}(a)
\eeqn where $e\in E$, $a\in \A_\theta$ and $i=1,2$.

{\it (b)} In our notation, $\delta_\tau =
  \frac{\delta_\omega}{\omega_2}$. Sending each $\nabla$ to
  $\nabla':=\frac{\nabla}{\omega_2}$ makes $\nabla'$ automatically compatible
  with $\delta_\tau$. More precisely, the functor sends
  $(\A_\theta,\delta_\omega)$ to $(\A_\theta,\delta_\tau)$ and $(E,\nabla)$ to
  $(E,\nabla')$.  
\end{proof}
We end this subsection with a technical result that we will need later on.
Let us define $\rwd(\tau) := \bar\tau\tau / |\Re(\tau)|$ as the `real width' of a transversal to $\tau \ZZ$ or, equivalently, the (interval) length of the intersection of a transversal to $\tau \ZZ$ with the reals $\RR \subset \CC$. Before we explain how this can be achieved, recall
that a {\it transversal} to $\tau \ZZ$ in $\CC$ is the image of a section of
the projection map $\CC \to \CC/\tau \ZZ$ ({\it e.g.,} the strip $0 \leq \Re
(z/\tau) < 1$).

\begin{lem}
\label{lem:tau-width}
For any $\tau \in \CC$ there is an element $g$ in $SL(2,\ZZ)$ such that $\rwd(g\tau) < 1$. 
\end{lem}
\begin{proof}
One simply computes that for $-(\tau +N)^{-1}$ (obtained as translation by an integer $N$ composed with inversion) we have
$$
\rwd( -(\tau+N)^{-1})  = \frac{1}{\Re(\tau) + N}.
$$
Thus, for any $\tau$ there exists an integer $N>0$ such that $\rwd( -(\tau+N)^{-1})<1$.
\end{proof}

\subsubsection{The derived category}


The derived category of holomorphic bundles is {\it defined} as the homotopy category of a
DG category (or a differential graded category), which is a category with the
$\Hom$ sets carrying a structure of a differential graded complex of
$\CC$-vector spaces (see \cite{Kel} for more details). The corresponding homotopy category is obtained by replacing the $\Hom$ complexes by their zeroth cohomologies. Henceforth, by $\TT_\theta^\tau$ we shall mean the noncommutative torus $\TT_\theta$ equipped with the derivation $\delta_\tau$. The DG category in consideration, denoted by $\mathcal{C}(\theta,\tau)$, consists of
objects of $\Vec(\TT^\tau_\theta)$, labelled by an integer
indicating its translation degree. The $\Hom$'s in $\mathcal{C}(\theta,\tau)$ are given by
a differential complex made up from the connections on the $\Hom$ sets in
$\Vec(\TT^\tau_\theta)$. Note that the $\Hom$ sets in $\Vec(\TT^\tau_\theta)$
also carry a module structure over some noncommutative torus (not necessarily $\TT_\theta$). 

Polishchuk and Schwarz constructed a functor from the
DG category $\mathcal{C}(\theta,\tau)$ to $D^b(X_\tau)$ and showed that the induced functor on the cohomology category is
fully faithful and that the image of $\Vec(\TT_\theta^\tau)$ lies in the
heart of the t-structure $(D^{\theta,\leqslant 0},D^{\theta,\geqslant})$, where 

\begin{eqnarray} \label{Poltstr}
D^{\theta,\leqslant 0}\; &=&\; \{K^\bullet\in D^b(X_\tau)\; |\; H^{>0}(K^\bullet) = 0, \text{ all ss factors of $H^0 (K^\bullet)$ have slope $>\theta$}\} \\ \nonumber
D^{\theta,\geqslant 0}\; &=&\; \{K^\bullet\in D^b(X_\tau)\; |\; H^{<-1}(K^\bullet) = 0, \text{ all ss factors of $H^{-1} (K^\bullet)$ have slope $\leqslant\theta$}\},
\end{eqnarray} with ss denoting semistable (see for the definition of slope Section \ref{StandBun} below).

Then Polishchuk showed that this functor actually induces an equivalence between $\Vec(\TT_\theta^\tau)$ and the heart $\mathcal{C}^{\theta,\tau}$ of the above t-structure \cite{Pol1}, whose derived category is again equivalent to
$D^b(X_\tau)$ \cite{Pol2}. This implies that $\Vec(\TT_\theta^\tau)$ is abelian and its derived category is
equivalent to $D^b(X_\tau)$ via the Polishchuk--Schwarz functor, denoted by $\mathcal{S_\tau}:
H^0\mathcal{C}(\theta,\tau)\map D^b(X_\tau)$.

\begin{rem} \label{summary}
The functor $\mathcal{S}_\tau$ actually induces an equivalence between
$\Vec(\TT_\theta^\tau)$ and $\mathcal{C}^{-\theta^{-1},\tau}$. Observe that $\left(\begin{smallmatrix}
                   0 & 1 \\
                   -1 & 0 \\
                 \end{smallmatrix}\right)
                  \theta = -\theta^{-1}$ says that $A_{-\theta^{-1}}$
                  is Morita equivalent to $\A_\theta$.
\end{rem}

\noindent
Summarising, one has the following equivalences of categories
\begin{equation} \label{QcohVsCoh}
\text{$\Vec(\TT_\theta^\tau)\cong \Vec(\TT_{-\theta^{-1}}^\tau) \cong
\mathcal{C}^{\theta,\tau}$ and $D^b(\mathcal{C}^{\theta,\tau})\cong
D^b(X_\tau)$}.
\end{equation}


\subsection{Standard bundles over $\TT_\theta^\tau$} \label{StandBun}

In $\Vec(\TT_\theta^\tau)$ there are certain distinguished objects called {\it standard holomorphic bundles} whose images under the Polishchuk--Schwarz functor $\mathcal{S}_\tau$ are {\it stable} objects inside $D^b(X_\tau)$, {\it i.e.,} objects of the form $\mathcal{F}[n]$ where $\mathcal{F}$ is a stable bundle on $X_\tau$ or the structure sheaf of a point \cite{PolSch}. The underlying $\A_\theta$-module of a standard holomorphic bundle is of the form $E_{m,n}$ with $m,n$ coprime as constructed in \cite{Con80}. One defines its {\it degree} and {\it rank} as
$$
\deg(E_{m,n}) := m;  \qquad
\rk(E_{m,n}) := m\theta + n
$$
whereas the {\it slope} of the bundle is defined as $\mu(E_{m,n}) := \frac{m}{m\theta +n}$.

If $m\neq 0$, given any $z\in\CC$ one endows $E_{m,n}$ with a connection of the form 

\beqn 
\nabla_z(f) =\frac{\partial f}{\partial x} + 2\pi i (\tau\mu(E_{m,n})x +z)f
\eeqn
If $m=0$, for a given $z\in\CC$ one equips $E_{0,1} =\A_\theta$ with the connection
$\nabla_z(a) = \delta_\tau(a) + 2\pi iz.a
$.
A {\it standard} holomorphic bundle is a module $E_{m,n}$ with $m\theta + n>0$ and equipped with this special connection $\nabla_z$. They are labelled by two coprime integers and a complex parameter. We set $E_{m,n}^z = (E_{m,n},\nabla_z)$; the modules $E_{0,1}^z=(\A_\theta,\nabla_z)$ will be abbreviated to $E_1^z$.

\section{Equivariant coherent sheaves and $\Vec(\TT^\tau_\theta)$}


We observe that there is an honest action of $\theta\ZZ$ on $X_\tau$ and hence on
$\Coh(X_\tau)$. Indeed, the point $\theta\,
mod\, (\ZZ+\tau\ZZ)$ on $X_\tau$ lies on the real axis of the
fundamental domain of the torus and its action is restricted to the circle
obtained by folding this axis. In fact, the action given by translations of $\theta$ on
$X_\tau$ transforms to the  action of multiplication by powers of $q=e^{2\pi i\theta}$ under the
Jacobi uniformization, {\it i.e.,} $z\longmapsto qz$ on
$\CC^\ast/\tilde{q}^\ZZ$, $\tilde{q}=e^{2\pi i\tau}$. So we are
confronted with a double quotient problem, where the actions commute. Namely,
it is the improper action of the group $q^\ZZ$ on $X_\tau$, which is itself obtained
by the free and proper action of the group
$\tilde{q}^\ZZ$ on $\CC^*$ (both actions are by
multiplication). Soibelman and Vologodsky have described the {\it quotient space} $\CC^*/q^\ZZ$ in terms of their noncommutative elliptic
curves $\BB_q$ in \cite{SoiVol}. The category $\BB_q$
is nothing but the category of $q^\ZZ$-equivariant (analytic) coherent sheaves
on $\CC^*$ (or equivalently, the category of modules over the crossed product algebra
$\sheaf(\CC^*)\rtimes_q \ZZ$, which are finitely presentable over
$\sheaf(\CC^*)$). It follows from Lemma 3.2 of \cite{SoiVol} that for any
$M\in\BB_q$ the underlying $\sheaf(\CC^*)$-module is free. However, there are
interesting actions of $\theta\ZZ$ or $q^\ZZ$ on the free modules with
respect to which they are equivariant. Let us denote by $\alpha$ the induced
action by automorphisms of $\theta \ZZ$ on $\sheaf(\CC^*)$: 
$$
\alpha (f)(z)= f(q z); \qquad (z \in \CC^*,q=e^{2\pi i\theta}).
$$
Here, we have understood the notation $\alpha:=\alpha(1)$ for the generator of
$\ZZ$, so that $\alpha(n) = \alpha^n$. What is lacking in this picture is an
infinitesimal action in terms of $\delta_\tau$ and compatible connections,
which accounts for the remaining $\tau\ZZ$ quotient operation. To this end, we define a derivation on
$\sheaf(\CC^*)$ by $\delta = \tau z \frac{d}{dz}$. It is this infinitesimal
action by $\delta$ that will turn out to be the appropriate replacement for
the infinitesimal action of the group $\tau \ZZ$.


\subsection{The category $\BB_q^\tau$}
Our goal in this section is to define a category alluded to before, which is
somehow `in between' the categories $\Vec(\TT_\theta^\tau)$ introduced by
Polishchuk and Schwarz and $\BB_q$ by Soibelman and
Vologodsky. More precisely, we would like to construct a category $\BB_q^\tau$ that is
functorially related to both of these categories. At the same time, we would
like to stay as close as possible to the setting of the Riemann--Hilbert
correspondence. The discussion above motivates us to define the following
category as a description of the quotient of $\BB_q$ by the infinitesimal
action of $\tau\ZZ$.

\begin{defn}
\label{defnB}
The category $\mathcal{B}_q^\tau$  consists of triples $(M,\sigma,\nabla)$, where
\begin{itemize}
\item $M$ is a finitely presentable $\sheaf(\CC^*)$-module, {\it i.e.,} there is an exact sequence,
\beqn
\sheaf(\CC^\ast)^m\map\sheaf(\CC^\ast)^n\map M\map 0.
\eeqn
\item $\sigma$ is a representation of $\theta\ZZ$ on $M$ covering the action
  $\alpha$ of $\theta \ZZ$ on $\sheaf(\CC^*)$, {\it i.e.,}
$$
\sigma(m\cdot f) = \sigma(m)\cdot \alpha(f); \qquad (m \in M, f \in \sheaf(\CC^*)).
$$
\item $\nabla$ is a $\theta \ZZ$-equivariant connection on $M$ 
covering the
  derivation $\delta =\tau z\frac{d}{dz}$ on $\sheaf(\CC^*)$, {\it i.e.,} it satisfies,
\begin{align*}
\nabla(m \cdot f)&=\nabla(m) \cdot f + m \cdot \d (f),\\
\nabla(\sigma(m)) &=\sigma( \nabla(m)),
\end{align*}
for all  $m \in M, f \in \sheaf(\CC^*)$.
\end{itemize}
In addition, we impose that the connection $\nabla$ is a {\rm regular singular
  connection} on $M$, that is, there exists a module basis $\{ e_1, \ldots,
e_n \}$ of $M$ for which the holomorphic functions (on $\CC^*$) $z^{-1}A_{ij}$ ($i,j =1,\ldots,n$) defined by $A_{ij} e_j = \nabla(e_i)$ have simple poles at $0$. We call $A=\left(A_{ij}\right)$ the {\rm matrix of the connection} with respect to that module basis.

\medskip

\noindent The morphisms in this category are equivariant
$\sheaf(\CC^*)$-module maps that are compatible with the connections.
We will also write $M=(M,\sigma,\nabla)$ when no confusion can arise. For two objects $M$ and $N$ we denote by $\Hom_{\sheaf(\CC^*)}^{\theta\ZZ,\d}(M,N)$ the $\CC$-linear vector space of morphisms between them.
\end{defn} 
The uniqueness of the matrix $A=A_{ij}$ (after the choice of a module basis $\{ e_i
\}$ for $M$) is due to the fact that the modules $M$ in $\BB_q^\tau$ turn out
to be free as $\sheaf(\CC^*)$-modules. This was observed in \cite[Lemma
  2]{SoiVol} and used the fact that the sheaf $M \otimes_{\CC} \sheaf_{\CC^*}$
must be torsion free due to $\theta \ZZ$-equivariance. Hence it is locally
free on $\CC^*$ and thus a trivial vector bundle. It also follows from the
fact that a coherent sheaf admitting a connection is automatically locally
free. As a consequence the $\sheaf(\CC^*)$-module
of its global sections is free. This freeness as $\sheaf(\CC^*)$-modules can be translated into freeness as $\theta\ZZ$-equivariant $\sheaf(\CC^*)$-modules as follows. Suppose that $M \simeq V \otimes_\CC \sheaf(\CC^*)$ as $\sheaf(\CC^*)$-modules with $V$ a complex vector space. Via this identification there is an induced action of $\theta \ZZ$ on $V \otimes_\CC \sheaf(\CC^*)$ making this an isomorphism of $\theta \ZZ$-equivariant $\sheaf(\CC^*)$-modules.


\begin{rem}
The $\tau$-dependence in the above construction might seem artificial, since all categories $\BB_q^\tau$ are equivalent for all $\tau$. Our notation is motivated by the relation of $\BB_q^\tau$ with $\Vec(\TT_\theta^\tau)$ which is established below. An equivalent construction could be obtained by translating the $\tau$-dependence from the category $\BB_q^\tau$ to the functor to $\psi^*: \BB_q^\tau \to \Vec(\TT_\theta^\tau)$ of Proposition \ref{ourfunctor} below.
\end{rem}

Let $\BB^\tau$ denote the category of pairs $(V,\nabla)$ with $V$ a vector bundle on $\CC^*$ and $\nabla$ a regular singular connection on $V$ associated to $\delta =\tau z \frac{d}{dz}$. 
By the above remarks on the modules $M$ in $\BB_q^\tau$, there is a functor
from $\BB_q^\tau$ to $\BB^\tau$ which forgets the action of $\theta
\ZZ$. Thanks to Deligne \cite{Del70} (see also, for instance, Theorem 1.1 and
the paragraph after Remark 1.2 of \cite{Mal87}), we know that the
category $\BB^\tau$ is equivalent to the category of finite dimensional
representations of the fundamental group $\pi_1(\CC^*,z')\simeq \ZZ$ with a base point $z'$. This result motivates the regular singularity condition we have imposed on the connections in Definition \ref{defnB}.

In Section \ref{section:RH} we will enhance this Riemann--Hilbert
  correspondence to an equivariant version and show that a similar statement
holds for $\BB^\tau_q$. Let us first proceed to examine some of the properties
of $\BB_q^\tau$ and its relation with the other two categories, {\it viz.,} $\BB_q$ and $\Vec(\TT_\theta^\tau)$.
\begin{prop} \label{BQT}
The category $\BB_q^\tau$ is an abelian category. 
\end{prop}
\begin{proof}
It is proven in Proposition 3.3 of \cite{SoiVol} that the category
$\BB_q$ is abelian. One observes readily that there is a faithful functor
(forgetting the connection) from $\BB_q^\tau$ to $\BB_q$. Suppose that $f: M
\to N$ is a morphism in $\BB_q^\tau$. Since it is also a morphism in $\BB_q$,
both $\ker f$ and $\coker f$ are equivariant
$\sheaf(\CC^*)$-modules. Moreover, the map $f$ intertwines the connections on
$M$ and $N$ and hence induces compatible connections on $\ker f$ and $\coker
f$ making them objects in $\BB_q^\tau$.
\end{proof}

\noindent
We now view $\A_\theta$ as a module over $\sheaf(\CC^*)$ via the homomorphism 

\begin{align*}
\psi: \sheaf(\CC^*) &\to \A_\theta\\
\sum_{n\in \ZZ} f_n z^n & \mapsto \sum_{n \in \ZZ} f_n U_1^n.
\end{align*} This is well-defined since a sequence $f_n$ of exponential decay is certainly
a Schwartz sequence. 

\begin{rem}
The map is essentially restricting a holomorphic function on $\CC^*$ to the unit circle. In fact, it is injective since, if a holomorphic
function vanishes on the unit circle, it must vanish on the whole of
$\CC^*$. Note that $\A_\theta$ is not finitely generated over $\sheaf(\CC^*)$
and hence not an element of $\BB_q$ or $\BB_q^\tau$.
\end{rem}
 
\begin{prop}
\label{ourfunctor}
The following association defines a right exact functor, denoted $\psi_*$, from $\BB_q^\tau$ to $\Vec(\TT_\theta^\tau)$. For an object $(M,\sigma,\nabla)$ in $\BB_q^\tau$
we define an object $(\tilde M, \tilde \nabla)$ in $\Vec(\TT_\theta^\tau)$ by
\begin{align*}
\tilde M      &= M \otimes_{\sheaf(\CC^*)} \A_\theta \\ 
\tilde \nabla &= 2 \pi i  ~\nabla \otimes 1 + 1 \otimes (\tau \delta_1 + \delta_2)\\
&=2 \pi i ~\nabla \otimes 1 + 1 \otimes \delta_\tau.
\end{align*}
\end{prop}
\begin{proof}
Observe that $\psi( 2 \pi i \delta f) = \tau \delta_1 (\psi(f)) $ as follows from the definitions of $\delta$ and $\delta_1$. 
Moreover, the image of $\sheaf(\CC^*)$ under the map
$\psi$ lies in the kernel of the derivation $\delta_2$ on the noncommutative
torus (since $\delta_2(U_1)$ is vanishing). Hence one can add $\delta_2$ to
$\tau\delta_1$ making $\tilde \nabla$ a connection on $\tilde M$ covering
$\delta_\tau$. 
\end{proof} 

\noindent
Note that by a simple adjustment one can actually define a right exact functor from $\BB^{\omega_1}_q$ to $\Vec(\TT_\theta^\omega)$.
We also claim that, in fact,

\begin{prop}
The map $\psi$ endows $\A_\theta$ with a flat module structure over $\sheaf(\CC^*)$. 
\end{prop}

\begin{proof}
The algebra $\sheaf(\CC^*)$ is a commutative integral domain, since holomorphic
functions cannot have disjoint support. Further, from Corollary 3.2 of
\cite{Pir} one concludes that the global $\Ext$ dimension of $\sheaf(\CC^*)$ is $1$. Hence it is a {\it Pr\"ufer
domain}, {\it i.e.,} a domain in which all finitely generated non-zero ideals
are invertible. Indeed, Theorem 6.1 of \cite{FucSal} says that a (fractional)
ideal in a domain is invertible if and only if it is projective and, since
$\sheaf(\CC^*)$ has $\Ext$ dimension $1$, given any finitely generated ideal
$I$, applying $\Hom(-,M)$ to the exact sequence $0\map I \map R\map
R/I\map 0$ for an arbitrary $M$, one finds that $\Ext^1(I,M)=0$, {\it i.e.,} $I$
is projective. It is known that a module over a Pr\"ufer domain is flat if
and only if it is torsion free (see, {\it e.g.,} Theorem 1.4 {\it ibid.}). So we
only need to check torsion freeness. We identify $\A_\theta$ as a module over
$\sheaf(\CC^*)$ with $\cS(\ZZ,C^\infty(\SS^1))$ and represent each
element as a sequence $\{g_n\}_{n\in\ZZ}$, $g_n \in C^\infty(\SS^1)$, refer to the discussion in Section \ref{HolCat}. The image of the map $\psi$ clearly lies
in $C^\infty(\SS^1)$, which is identified with the functions supported at the
identity element of $\ZZ$. In other words, for all $f \in
\sheaf(\CC^*)$, $\psi(f)$ is of the form $\{f_n\}$, where $f_n=0$ unless $n=0$. Now consider any $g=\{g_n\}\in \A_\theta$ and suppose that some 
non-zero $f\in \Ann(\{g_n\})$, {\it i.e.,} $g\ast \psi(f) =\{g_n\alpha_n(f_0)\} = 0$. This implies that $g_n(z)f_0(q^{n}z) = 0$ for all $n$, $|z|=1$. Being the
restriction of a holomorphic function on $\CC^*$, $f_0(q^{n}z)$ has a discrete
zero set on the unit circle. A smooth function on $\SS^1$ cannot have a
discrete set of points as support and hence each $g_n(z)$ must be identically zero. Thus, whenever an element in $\A_\theta$ has a
non-zero element in its annihilator ideal, the element is itself zero. Hence
$\A_\theta$ is torsion free from which the result follows. \end{proof}

\begin{cor}
The base change functor $\psi_*$ induced by the homomorphism $\psi$ is exact
and faithful.
\end{cor}

\begin{proof}
From the previous Proposition we conclude that the functor sends an exact
sequence of $\sheaf(\CC^*)$-modules to an exact sequence of $\A_\theta$-modules and clearly the induced morphisms respect the connections. For the faithfulness, identify each object $M\in \BB_q^\tau$ with
$V\otimes\sheaf(\CC^*)$ with $V$ a vector space; similarly write $M' = V'\otimes\sheaf(\CC^*)$. A morphism in $\BB_q^\tau$ from $M$ to $M'$ is then given by an element in $\Hom_\CC(V,V') \otimes_\CC \sheaf(\CC^*)$, whereas a morphism in $\Vec(\TT_\theta^\tau)$ between $\tilde M$ and $\tilde M'$ is given by an element in $\Hom_\CC(V,V') \otimes_\CC \A_\theta$. The functor $\psi_*$ acts on these element by $1 \otimes \psi$ and since $\psi$ is injective, it follows that $\psi_*$ is injective on morphisms. 
\end{proof}

\begin{rem} 
\label{ImageFunctor}
However, the functor is not full. It is certainly not essentially surjective
as the underlying $\A_\theta$-modules of the objects in the image are all
free, whilst $\Vec(\TT_\theta^\tau)$ has modules which are not free. It is also not injective on objects.
\end{rem}
\begin{rem}
\label{functorBtau}
Before Proposition \ref{BQT} we introduced the category $\BB^\tau\simeq \Rep(\ZZ)$ as the category of bundles on $\CC^*$ with regular singular connections compatible with $\tau z\frac{d}{dz}$. The functor $\psi_*$ can actually be regarded as a functor between $\BB^\tau$ and $\Vec(\TT_\theta^\tau)$ and we may precompose it with the forgetful functor $\BB_q^\tau\map\BB^\tau$ to get our desired $\psi_*:\BB_q^\tau\map\Vec(\TT_\theta^\tau)$. 
\end{rem}

\noindent
The main Theorem of \cite{Pol1} says that the category generated by successive
extensions of all standard holomorphic bundles over $\TT_\theta^\tau$ is already all of $\Vec(\TT_\theta^\tau)$ (refer to Subsection \ref{StandBun} for the definition of standard bundles). 
Let us denote the full subcategory of $\Vec(\TT_\theta^\tau)$ generated by successive
extensions of standard modules of the form $E^{z'}_1$,
$z'\in\CC$ by $\FrVec(\TT_\theta^\tau)$. Since the extension of two free
modules is again free, it is clear that the underlying $\A_\theta$-module of
all objects of $\FrVec(\TT_\theta^\tau)$ is free.

\begin{lem} \label{FrVect}
With respect to a suitable basis each object of $\FrVec(\TT_\theta^\tau)$ is of the form
$(\A_\theta^n,\delta_\tau + A)$, where $A$ is an $n\times n$ upper triangular
matrix in $M_n(\A_\theta)$ with diagonal entries in $\CC$. 
\end{lem}

\begin{proof}
It is known that given any finitely generated projective module $M$ over
$\A_\theta$ and a fixed connection $\nabla$ compatible with $\delta_\tau$, all
other compatible connections are of the form $\nabla + \phi$, $\phi\in
\End_{\A_\theta}(M)$. This follows easily from the Leibniz rule
\eqref{leibniz}. Since $M$ is of the form $\A_\theta^n$, $\phi$ is determined by a matrix $A\in M_n(\A_\theta)$. Let 

\begin{equation*} 
0\map (\A_\theta,\nabla_{z'})\overset{\iota}{\map} (\A_\theta^2,\delta_\tau + A)\overset{\pi}{\map}
(\A_\theta,\nabla_{z''})\map 0
\end{equation*} 
be a holomorphic extension in $\Vec(\TT_\theta^\tau)$. Write $A
= \left(\begin{smallmatrix}
        a & b \\
        c & d \\
        \end{smallmatrix}\right)$ with entries $a,b,c,d\in\A_\theta$ and 
        $\iota(a) =(a,0)$ and $\pi(a_1,a_2)=a_2$. One checks easily that the holomorphicity of
        $\iota$ and $\pi$ (the fact that they commute with the connections)
        forces $c = 0, a = z'$ and $d = z''$. Now by induction it follows that
        the connections obtained by successive extensions are of the desired
        form. 

Conversely, by induction suppose that every connection
        of the desired form on $\A_\theta^{n-1}$ can be obtained as an iterated extension of modules of the form $E^{z'}_1$. Let $A$ be an upper triangular
        matrix in $M_n(\A_\theta)$ whose diagonal entries are in $\CC$, {\it
        i.e.,} $A$ is of the form
$$
\begin{pmatrix}
z' & b_2 &\cdots &b_n \\
0 & &&\\
\vdots && A' &\\
0 &&&\\
\end{pmatrix},
$$
where $A' \in M_{n-1}(\A_\theta)$ is also of the prescribed type and $b_2, \ldots, b_n \in \A_\theta$. A routine calculation then shows that 
\begin{align*}
0\map (\A_\theta,\nabla_{z'}) \overset{\iota}{\map} (\A_\theta^n,\delta_\tau + A)
\overset{\pi}{\map} (\A_\theta^{n-1},\delta_\tau + A')\map 0
\end{align*}
with $\iota(a) = (a,0,\dots,0)$ and $\pi(a_1,a_2,\dots,a_n) =
(a_2,\dots,a_n)$ is a holomorphic extension in $\Vec(\TT_\theta^\tau)$. Hence
$(\A_\theta^n,\delta_\tau + A)$ belongs to $\FrVec(\TT_\theta^\tau)$. 
\end{proof}

                                                           
\begin{rem} \label{ConstEnt}
Given any matrix $A\in M_n(\CC)$, with respect to a suitable basis one can
reduce it to its Jordan canonical form (it is also upper triangular with
diagonal entries in $\CC$). Therefore, $\FrVec(\TT_\theta^\tau)$ contains all
objects of the form $(\A_\theta^n,\delta + A)$, where $A\in M_n(\CC)$ with
respect to a basis. 
\end{rem}

\noindent

As we will see later (Proposition \ref{ConnMatrix}), each object
$(M,\sigma,\nabla)$ in $\BB_q^\tau$ is isomorphic to an object, whose matrix of the connection is a constant matrix. This can be accomplished via a change of basis of $M$. Combining this with the above remark, we conclude that the image of $\psi_*$ is a subcategory $\FrVec(\TT_\theta^\tau)$.

\subsection{The effect on {K}-theory}

\noindent
We infer from Eqn. \eqref{QcohVsCoh} that the K-theory (by that we mean
the Grothendieck group, {\it i.e.,} the free abelian group generated by the
isomorphism classes of objects modulo the relations coming from all exact sequences) of $\Vec(\TT_\theta^\tau)$ is isomorphic to that of $D^b(X_\tau)$
via the Polishchuk--Schwarz equivalence $\mathcal{S}_\tau$. One knows that
$K_0(D^b(X_\tau))\cong K_0(\Coh(X_\tau)) = K_0(X_\tau) = \Pic(X_\tau)\oplus \ZZ$. The composition of the functors $\psi_*$ followed by
  $\mathcal{S}_\tau$ induces a homomorphism between $K_0(\BB_q^\tau)$ and
$K_0(\Vec(\TT_\theta^\tau))=\Pic(X_\tau)\oplus \ZZ$. One observes that
  applying $\psi_*$ one obtains only elements in $\Vec(\TT_\theta^\tau)$ whose
  underlying $\A_\theta$-modules are free. It is known that for $E\in
  \Vec(\TT^\tau_\theta)$, $\rk\mathcal{S}_\tau (E) =
  -\deg(E)$ and $\deg\mathcal{S}_\tau (E) = \rk (E)$. 
The degree of the modules, which are free, is known to be zero. Hence the composition of the two functors sends every element in
  $\BB_q^\tau$, whose image under $\psi_*$ is a standard bundle, to a torsion sheaf on $X_\tau$. One can check that
  $\sheaf(\CC^*)$ equipped with the connection $\d +  z'$, where
  $z'\in \CC$, gets mapped to the standard holomorphic bundle
  $E_1^{z'}$ as explained after Remark \ref{functorBtau}. From part (c) of Proposition 3.7 of
  \cite{PolSch} we know that $\mathcal{S}_\tau(E_{1}^{z'})$ is
  $\sheaf_{-z'}$ (up to a shift in the derived category), which
  is the structure sheaf of the point $-z'$ mod $(\ZZ + \tau\ZZ)$ in
  $X_\tau$. All modules of the form 
  $(\sheaf(\CC^*),\sigma, \d + z')$ with $z'\in\CC$ are endomorphism simple, {\it i.e.,}
  $\End(\sheaf(\CC^*),\sigma, \d + z') = \CC$. Indeed, ignoring the equivariance condition and
  the connection, $\End(\sheaf(\CC^*))=\sheaf(\CC^*)$ and the equivariance condition
  says that $\sigma(mf)=\sigma(m)f$. However, by definition $\sigma(mf) =
  \sigma(m)\alpha(f)$ whence $\alpha(f) = f$ implying $f\in\CC$. This
  module is mapped to $(\A_\theta,\delta_\tau + 2\pi iz') = E_1^{z'}$, which in
  turn is mapped to the endomorphism simple object $\sheaf_{-z'}$ in $\mathcal{C}^{\theta,\tau}$. It is known that, in fact, the Grothendieck group of any
  nonsingular curve $C$ is isomorphic to $\Pic(C)\oplus \ZZ$. In this
  identification the contribution to $\ZZ$ comes from the rank of the coherent sheaf, whereas $\Pic(C)$ can be regarded as the contribution from the
  torsion part (actually from the determinant bundle of the sheaf, which may
  be identified with a torsion sheaf via a Fourier-Mukai transform). Since we shall see later on (Corollary 21) that the classes of $(\sheaf(\CC^*),\sigma,\d + z')$ generate the K-theory of $\BB_q^\tau$, the image of the induced map on K-theory lies inside $\Pic(X_\tau)$.

\begin{prop} \label{K-Theory}
The map induced by $\mathcal{S}_\tau\circ\psi_*$ between the K-theories of
$\BB_q^\tau$ and $\Vec(\TT_\theta^\tau)$ gives a surjection from $K_0(\BB_q^\tau)$ to $\Pic(X_\tau)$.
\end{prop}

\begin{proof}
The divisor class group of $X_\tau$ is the free abelian group generated by the
points of $X_\tau$ modulo the principal divisors, which is also isomorphic to
$\Pic(X_\tau)$. The class of each point $z'\in X_\tau$ of the divisor class
group can be identified with the class of the torsion sheaf $\sheaf_{z'}$ corresponding to the line bundle $\sheaf(z')\in
\Pic^1(X_\tau)$ and they generate $\Pic(X_\tau)$ as a group. By the above
argument $\sheaf_{z'}$ is obtained by applying the functor $\mathcal{S}_\tau\circ\psi_*$ to the
element $(\sheaf(\CC^*) ,\sigma,\delta - z')$ of $\BB_q^\tau$. Thus one obtains
a surjection onto the generating set of $\Pic(X_\tau)$ from which the
assertion follows.
\end{proof}

\begin{rem}
From Proposition 2.1 of \cite{PolSch} we know that the images of $(\sheaf(\CC^*),\sigma ,\delta +z'_1)$ and $(\sheaf(\CC^*),\sigma ,\delta + z'_2)$ under $\psi_*$ are isomorphic if and only if $z'_1\equiv z'_2 \,\mod \,
(\ZZ+\tau\ZZ)$. More generally, abbreviating the module
$(\sheaf(\CC^*),\sigma,\delta + z')$ by $M_{z'}$, one can also rephrase the linear equivalence
relation of the divisor class group to conclude that an element of the form $\sum
n_i[M_{-z'_i}]$ maps to zero at the level of K-theory whenever $\sum n_i =0$ and $\sum n_iz'_i\in
(\ZZ+\tau\ZZ)$. However, some of them actually represent the trivial class in
the K-theory of $\BB_q^\tau$, as we will see in the next section (see
Corollary \ref{ktheory}).
\end{rem}

\bigskip

Although the image of $\BB_q^\tau$ gives only the free modules in $\Vec(\TT_\theta^\tau)$, it has the interesting property of being a Tannakian category, as we will explore in the next section. Let us end this section by summarising the relations between $\BB_q^\tau$ and the categories $\BB^\tau$, $\BB_q$, $\Vec(\TT_\theta^\tau)$:
$$
\xymatrix{&\BB_q^\tau  \ar[dl] \ar[d]^{\psi_*} \ar[dr]& \\ \BB_q & \Vec(\TT_\theta^\tau) & \BB^\tau}
$$
where the two diagonal arrows are the forgetful functors discussed before. All
of these functors are faithful and exact.

\section{The Tannakian formalism and the equivariant Riemann--Hilbert correspondence}
\label{section:RH}
We will now analyse further the structure of $\BB_q^\tau$ and define a tensor product on it. Our main result is that this -- together with a fibre functor -- makes $\BB_q^\tau$ a Tannakian category. Via an equivariant version of the Riemann--Hilbert correspondence on $\CC^*$, we determine the corresponding affine group scheme. 

\subsection{Preliminaries on {T}annakian categories}
We briefly recall the notion of a Tannakian category. For more details, we refer the reader to the original works \cite{Saa, Del90, DelMil} (see also Appendix B of \cite{PutSing}). 

Let $\C$ be an $k$-linear abelian category, for a field $k$. Then $\C$ is a neutral Tannakian category over $k$ if 
\begin{enumerate}
\item The category $\C$ is a {\it tensor category}. In other words, there is a tensor product: for every pair of objects $X,Y$ there is an object $X \otimes Y$. The tensor product is commutative $X \otimes Y \simeq Y \otimes X$ and associative $X \otimes (Y \otimes Z) \simeq (X \otimes Y) \otimes Z$ and there is a unit object $1$ (such that $X \otimes 1 \simeq 1 \otimes X \simeq X$). The above isomorphisms are supposed to be functorial. 
\item $\C$ is a {\it rigid} tensor category: there exists a duality $\vee: \C \to \C^\op$, satisfying
\begin{itemize}
\item For any object $X$ in $\C$, the functor $_- \otimes X^\vee$ is left adjoint to $_- \otimes X$, and the functor $X^\vee \otimes _-$ is right adjoint to $X \otimes _-$. 
\item There is an evaluation morphism $\epsilon: X \otimes X^\vee \to 1$ and a unit morphism $\eta: 1 \to X^\vee \otimes X$ such that $(\epsilon \otimes 1) \circ (1 \otimes \eta) = 1_X$ and $(1 \otimes \epsilon) \circ (\eta \otimes 1) = 1_{X^\vee}$. 
\end{itemize}
\item An isomorphism between $\End(1)$ and $k$ is given. 
\item There is a fibre functor $\omega: \C \to \Vec_k$ to the category of $k$-vector spaces: this is a $k$-linear, faithful, exact functor that commutes with tensor products. 
\end{enumerate} 
In general, the fibre functor could be $L$-valued, where $L$ is a field extension of $k$. Henceforth, unless otherwise stated, we shall consider only {\it neutral} Tannakian categories, {\it i.e.,} those with a $k$-valued fibre functor. An important result is that every Tannakian category is equivalent to the
category of finite dimensional linear representations of an affine group
scheme $G$ over $k$. Abstractly, it is given by the automorphism group scheme of the fibre functor. However, in most examples in the literature, a concrete equivalence with $\Rep(G)$ for some $G$ is established, and we will do so as well.

\subsection{The {T}annakian category structure on $\BB_q^\tau$}
Let us generalize a little and let $(R,\delta)$ be a differential (commutative) ring that carries
an action $\alpha$ of a group $G$. Let $\Mod^{G,\delta}(R)$ denote the
category consisting of free $G$-equivariant differential $R$-modules.
Recall that a differential $R$-module is an $R$-module equipped with a map
$\nabla: M \to M$ -- a connection -- that satisfies the Leibniz rule: 
\begin{align*}
\nabla(m \cdot r)&=\nabla(m) \cdot r + m \cdot \delta (r).
\end{align*}
Moreover, $G$-equivariance means that there is an action $\sigma$ of $G$ such that 
\begin{align*}
\sigma_g(m\cdot r)&=\sigma_g(m)\cdot \alpha_g(r),\\
\nabla(\sigma_g(m)) &=\sigma_g( \nabla(m)).
\end{align*}
We will group the objects in the category $\Mod^{G,\delta}(R)$ into a triple
$(M,\sigma,\nabla)$ and denote the morphisms that respect all the structures by $\Hom^{G,\delta}_R(M,N)$.

\begin{prop}
\label{prop-rigid}
The category $\Mod^{G,\delta}(R)$ is a rigid tensor category with the tensor product given by
$$
(M,\sigma,\nabla) \otimes (N,\sigma',\nabla') 
          = ( M \otimes_{\sheaf(\CC^*)} N, ~ \sigma \otimes \sigma', ~\nabla \otimes 1 + 1 \otimes \nabla' )
$$
for any two objects $(M,\sigma,\nabla)$ and $(N,\sigma',\nabla')$ in $\Mod^{G,\delta}(R)$. 
\end{prop}
\begin{proof}
We start by checking that the tensor product is commutative. First of all, since $R$ is a commutative ring, the `tensor flip' that maps $M \otimes_\CC N \to N \otimes_\CC M$ factorizes to a bijective map of $R$-modules from $M \otimes_{R} N$ to $N \otimes_{R} M$. One also checks that it intertwines the actions $\sigma \otimes \sigma'$ and $\sigma'\otimes \sigma$ and the two connections. 

The duality is given as follows, for an object $(M=V \otimes R,\sigma,\nabla)$, $V$ a vector space, we define its dual object $(M^\vee, \sigma^\vee, \nabla^\vee)$ as follows. Define an $R$-module by,
$$
M^\vee:=\Hom_{R}(M, R ),
$$
with $r \in R$ acting on $f \in M^\vee$ by $(f\cdot r) (m) = f(m)\cdot r = f(m\cdot r)$.
It can be equipped with a dual action $\sigma^\vee$ of $G$ by setting for $f \in M^\vee$, 
$$
\sigma^\vee( f ) = \alpha \circ f \circ \sigma^{-1}.
$$
One can check that $\sigma^\vee (f)$ is again $R$-linear:
$$
\sigma^\vee (f) (m\cdot r)= \alpha \circ f \left( \sigma^{-1}(m) \cdot \alpha^{-1}( r) \right)
= \alpha \circ f \circ \sigma^{-1} (m) \cdot r = \left( \sigma^\vee(f)\cdot r \right)(m)
$$
Moreover, the action of $R$ on $M^\vee$ is equivariant with respect to $\sigma^\vee$:
$$
\sigma^\vee (f \cdot r) (m) = \alpha \circ (f \cdot r) \left(\sigma^{-1} (m)\right)= 
  \alpha\left( f (\sigma^{-1} (m))  \cdot r \right)= \alpha \circ f \circ \sigma^{-1} (m) \cdot \alpha(r).
$$
A dual connection $\nabla^\vee$ is defined by 
$$
\nabla^\vee(f) = \delta \circ f  - f \circ \nabla,
$$
which indeed satisfies the Leibniz rule
$$
\nabla^\vee (f \cdot r) (m) = 
\delta (f(m)) \cdot r + f(m) \cdot \delta(r) - f (\nabla(m)) \cdot r  =
\left(\nabla^\vee(f)\cdot r\right)(m) + (f \cdot \delta(r))(m),
$$
and is $\sigma^\vee$-invariant:
$$
\sigma^\vee \left( \nabla^\vee(f) \right) = \alpha \circ (\delta \circ f) \circ \sigma^{-1} - \alpha \circ (f \circ \nabla) \circ \sigma^{-1} 
=  \delta  \circ (\alpha \circ f \circ \sigma^{-1} )  - (\alpha \circ f  \circ \sigma^{-1}) \circ \nabla,
$$
since $\alpha$ and $\sigma$ commute with $\delta$ and $\nabla$, respectively. 

Note that since $M=V\otimes R$, we can identify,
\begin{align*}
M^\vee \simeq \Hom_R(V\otimes R, R )
\simeq  \Hom_\CC (V, \CC) \otimes R 
\simeq  V^* \otimes R,
\end{align*}
from which it follows that $M^{\vee\vee} \simeq M$. Indeed, one checks that the induced map respects the extra $(G,\delta)$-structure:
\begin{gather*}
\sigma^{\vee\vee}(m)(f) = \alpha \circ m \circ \left(\sigma^\vee\right)^{-1} (f) =\alpha \circ m \circ (\alpha^{-1} \circ f \circ \sigma)= f\left(\sigma(m)\right)\\
\nabla^{\vee\vee}(m) (f)=(\delta \circ m)(f) - m \circ \nabla^\vee(f) = \delta\left(f(m)\right) - m \circ (\delta \circ f) + f\left( \nabla(m) \right) = f(\nabla(m)).
\end{gather*}
for all $m \in M, f \in M^\vee$.
In addition, it allows one to prove that the association 
\begin{align*}
\phi \in \Hom_{R}^{G,\delta}(N_1, M^\vee \otimes_{R} N_2 ) &\mapsto \tilde \phi \in \Hom_R^{G,\delta}(M \otimes_{R} N_1, N_2) \\
\tilde \phi(m \otimes n_1) &:= \phi(n_1)(m) \in N_2.
\end{align*}
induces an isomorphism. Again, it is enough to show that this map is both $G$-equivariant and $\delta$-invariant, which is left as an exercise.

\noindent
In a similar way, one proves that 
\begin{align*}
\Hom_R^{G,\delta}(N_1 \otimes_{R} M^\vee, N_2) \simeq \Hom_{R}^{G,\delta}(N_1, N_2 \otimes_{R} M ).
\end{align*}
Finally, there is an evaluation morphism and a unit morphism given in terms of a basis $\{ e_i \}$ of $V$ and its dual $\{ \hat e_i \}$ of $V^*$ by
\begin{gather*}
\epsilon(m\otimes f) = f(m), \qquad \eta(1_R) = \hat e_i \otimes e_i,
\end{gather*}
that satisfy the required properties.
\end{proof}
Let us now return to the category $\BB_q^\tau$ of Definition \ref{defnB}.
It is not difficult to see that the above tensor product respects the regular
singularity condition in the definition of $\BB_q^\tau$. Hence this becomes a
rigid tensor category as well. We would like to show that it is in fact a
Tannakian category by constructing a fibre functor to $\Vec_\CC$. The following observations turn out to be essential in what follows.

Via a series of changes of basis, it is possible to bring the matrix $A$ in
the form of a constant matrix with all eigenvalues in the same transversal of
$\tau \ZZ$. In other words, its eigenvalues never differ by an
integer multiple of $\tau$. We follow the argument of Section 17 in \cite{Was}. Let $A(z)=A_0 + A_1 z + \cdots$ be a matrix with holomorphic entries. We first
bring the constant term $A_0$ in Jordan canonical form via a constant change
of basis matrix. Subsequently, we can bring all the eigenvalues of $A_0$ in
the same transversal of $\tau \ZZ$ by the so-called {\it shearing transformations}. Let us consider the case of a $2 \times 2$ matrix $A(z)$ and write
$$
A(z) = \begin{pmatrix} \lambda_1 & 0 \\ 0 & \lambda_2 \end{pmatrix} + \begin{pmatrix} a(z) & b(z) \\ c(z) & d(z) \end{pmatrix}, 
$$
with $a=a_1 z + a_2 z^2 + \cdots $ and similarly $b,c$ and $d$.
Let us suppose that $\lambda_1 - \lambda_2 = k \tau$ for some positive integer $k$. The change of basis is given by the matrix $D=\textup{diag}\{ 1, z \}$ and transforms $A$ to
$$
A' = D^{-1} A D + D^{-1} \delta D = \begin{pmatrix} \lambda_1 & 0 \\ c_1 & \lambda_2+\tau \end{pmatrix} + \begin{pmatrix} a(z) &z b(z)\\ c_2 z+ c_3 z^2 + \cdots  & d(z) \end{pmatrix},
$$ 
and one readily observes that the constant term $A'_0 $ of this matrix has eigenvalues that differ by $(k-1)\tau$. Proceeding in this way, one can transform $A$ to a matrix that has constant term with eigenvalues in the same transversal. The generalization to arbitrary dimensions is straightforward and can be found in Section 17.1 of {\it loc. cit.}

\begin{prop} \label{ConnMatrix}
For each object in $\BB_q^\tau$ there is an isomorphic object $(M = V \otimes \sheaf(\CC^*),\sigma,\nabla)$ in $\BB_q^\tau$ with $V$ a vector space and
\begin{enumerate}
\item $\nabla =\delta + A$ with $A$ a constant matrix with all eigenvalues in the same transversal of $\tau \ZZ$,
\item $\sigma$ is given by $\sigma(v \otimes f) = B v \otimes \alpha (f)$ for an invertible constant matrix $B$.
\end{enumerate}
\end{prop}

\begin{proof}
Since $M$ is a free $\sheaf(\CC^*)$-module, there is a vector space $V$ such that $M \simeq V \otimes \sheaf(\CC^*)$. We show {\it 1.} by adopting an argument from Section 5 of \cite{Was}. By the above observations, we can write the matrix of the connection as $A=A_0 + A_1 z
+ \cdots$, with $A_0$ having eigenvalues that never differ by an element of
$\tau\ZZ$. We construct a matrix $P=\I +
P_1 z + \cdots $ ($P_k$ in $M_n(\CC)$) which solves $P A_0 = AP - \delta
P$. Comparing the powers of $z$, we find
$$
A_0 P_k - P_k (A_0 + \tau k \mathrm{I} ) = - (A_k + A_{k-1} P_1 + \cdots + A_1 P_{k-1} )
$$
which can be solved recursively by our assumption on the eigenvalues of
$A_0$. This gives a formal power series expansion and we would like to show
that the entries of $P$ are in fact holomorphic functions on $\CC^*$.

Now by Theorem 5.4 of \cite{Was}  one knows that the radius of convergence of the
entries of $P$ is the same as that of the entries of $A$, which is
infinity. Hence, $P\in M_n(\sheaf(\CC^*))$. 

Next, the action of $\sigma$ can be written as $\sigma(v \otimes f)=Bv \otimes \alpha(f)$ for some invertible matrix $B \in M_n(\sheaf(\CC^*))$ with $n$ the dimension of $V$. Expressed in terms of $A$ and $B$, the equivariance condition $\sigma \circ \nabla = \nabla \circ \sigma$ reads
\begin{equation} \label{compatibility}
\delta B + [A,B] = 0 ,
\end{equation} and as observed above, we may assume that $A$ has constant entries and with eigenvalues that are all in the same transversal. We adopt the argument from the proof of Theorem 4.4 in \cite{Mal87} to show that $B$ is in fact constant. Writing $B$ as a Laurent series $B=\sum_{k \in \ZZ} B_k z^k$ we obtain the following relations
$$
(A -\tau k \textup{I}_n)B_k = B_k A, \qquad k \in \ZZ.
$$ 
This implies \cite[Theorem 4.1]{Was} (see also Lemma 4.6 in \cite{Mal87}) that $(A- \tau k \textup{I}_n )$ and $A$ have at least one common eigenvalue. But since the eigenvalues of $A$ are all in a transversal of $\tau \ZZ$ in $\CC$, this is impossible unless $k=0$, and we conclude that $B_k =0$ for all $k \neq 0$. 
\end{proof}

Our next task is to show that $\BB_q^\tau$ is in fact a Tannakian category and compute the corresponding affine group scheme. For this, we use an equivariant version of the Riemann--Hilbert correspondence. 
\begin{thm}
\label{thm:RH}
\begin{enumerate} 
\item \label{part1} The category $\BB_q^\tau$ is a Tannakian category with the fibre functor given by 
\begin{align*}
\omega: \BB_q^\tau & \map \Vec_\CC\\
 (M,\sigma,\nabla) &\longmapsto (\ker \nabla)_z,
\end{align*}
mapping to the germs at a fixed point $z\in\CC^*$ of local solutions to the
differential equation $\delta f + A f = 0$, where $\nabla = \delta + A$ with
respect to a suitable basis of $M$. 
\item \label{part2} The category $\BB_q^\tau$ is equivalent to the category $\Rep(\ZZ + \theta \ZZ)$ of finite dimensional representations of $\ZZ + \theta \ZZ \simeq \ZZ^2$.
\end{enumerate}
\end{thm}
\begin{proof}
By the existence and uniqueness of local
solutions of linear differential equations, there are $n$ local solutions to
the system of differential equations $\delta U = -A U $ once we have fixed the
initial conditions, so that $(\ker \nabla)_z$ is an $n$-dimensional complex vector space. That the functor $\omega$ is faithful can be seen as follows. Suppose $\phi$ is a morphism between two objects $(M,\sigma,\nabla)$ and $(M',\sigma',\nabla')$ and suppose that these objects are of the form as in Proposition \ref{ConnMatrix}, with the eigenvalues of $A,A'$ in the same transversal. We claim that $\phi$ is given by a constant matrix so that $\omega(\phi)$ mapping $(\ker \nabla)_z$ to $(\ker \nabla')_z$ coincides with $\phi$. 
The argument is very similar to that used in the second part of the proof of Proposition \ref{ConnMatrix} since compatibility of $\phi$ with the connections implies
\begin{align*}
(A' -\tau k \textup{I}_n)\phi_k = \phi_k A,\qquad  k \in \ZZ,
\end{align*}
where we have written $\phi = \sum_{k \in \ZZ} \phi_k z^k$.
An application of Theorem 4.1 in \cite{Was} then implies that $A$ and $A'-\tau k \textup{I}_n $ have a common eigenvalue. This is impossible unless $k=0$ since by assumption $A$ and $A'$ have eigenvalues in the same transversal. We conclude that $\phi_k =0$ for all $k\neq 0$ so $\phi$ is given by a constant matrix, intertwining $A$ and $A'$.

The general case follows by observing that Proposition \ref{ConnMatrix} implies that a morphism between two objects in $\BB_q^\tau$ can always be written as $D_2 \circ \phi \circ D_1^{-1}$ with $\phi$ constant as above and with $D_i$ certain (invertible) change of basis matrices. 

\medskip

For {\it 2.}, fix a transversal $\trans$ to $\tau \ZZ$ in $\CC$. We construct a tensor functor $\F_\trans: \Rep(\ZZ^2) \to \BB_q^\tau$ that is full, faithful and essentially surjective. 
Let $\rho_1, \rho_2$ be two mutually commuting representations of
$\ZZ$ on a vector space $V$. Then we define $A \in \End(V)$ via $\rho_1(1) = e^{2
  \pi i  A/\tau}$ and $B$ as $\rho_2(1)$. By Lemma 4.5 in \cite{Mal87}, there
exists a unique matrix $A$ such that $\rho_1(1) = e^{2\pi i A/\tau}$ with its eigenvalues in the transversal $\trans$ and a unique matrix $B'$ such that $B = e^{2\pi i B'}$. We set $\F_\trans(V)=(M,\sigma,\nabla)$ in $\BB_q^\tau$ by setting $M=V \otimes \sheaf(\CC^*)$, $\sigma(v
\otimes f) = Bv \otimes \alpha(f)$ and finally $\nabla(v \otimes f) =  A v
\otimes f + v \otimes \delta f$; for a morphism $\phi \in \Hom(V,V')$ we simply set $\F_\trans(\phi) = \phi \otimes 1$.
Once again by Lemma 4.5 {\it ibid.} the
matrices $A$ and $B'$ commute, whence $A$ and $B=e^{2\pi i B'}$ commute. Thus
the compatibility condition between $\sigma$ and $\nabla$ given by Eqn.
\eqref{compatibility} is satisfied. Moreover, $\F_\trans(\phi)$ is compatible with $\sigma$ and $\nabla$ and thus a morphism in $\BB_q^\tau$. 

We infer from Proposition \ref{ConnMatrix} that the functor $\F_\trans$ is essentially surjective, since any object in $\BB_q^\tau$ is isomorphic to an object obtained from an element in $\Rep(\ZZ^2)$ by the above procedure.

Fullness and faithfulness of this functor can be seen as follows. Let $V,V'$ be two
vector spaces with the action of $\ZZ^2$ given by $e^{2\pi i A/\tau}, B$ and
$e^{2\pi i A'/\tau}$, $B'$ respectively. We can choose the square matrices $A$ and $A'$ such that their eigenvalues lie in the transversal $\trans$. It then follows by the same reasoning as before that an element $\rho \in \Hom_{\sheaf(\CC^*)}^{\theta \ZZ,\delta}(M,M')$ is given by a constant matrix that intertwines $A,B$ and $A',B'$, respectively. Hence, it is given by an element in $\Hom(V,V')$ that commutes with $\rho_1$ and $\rho_2$ ({i.e.} a morphism in $\Rep(\ZZ^2)$).

Finally, we show that $\F_\trans$ is a tensor functor. 
Suppose that $(V,\rho_1,\rho_2)$ and $(V',\rho_1',\rho_2')$ are two objects in
$\Rep(\ZZ^2)$; we need to show that there are natural isomorphisms $c_{V,V'}: F(V) \otimes F(V') \to F(V \otimes V')$. As before, we define the connection matrix $A$ by setting $e^{2\pi i A/\tau} = \rho_1(1)$ and $B= \rho_2(1)$; in the same manner we define $A'$ and $B'$ from $\rho_1'$ and $\rho_2'$. We then have
\begin{align*}
F(V,\rho_1,\rho_2) \otimes F(V',\rho_1',\rho_2') = \left( (V\otimes\sheaf(\CC^*) )\otimes_{\sheaf(\CC^*)} (V'\otimes\sheaf(\CC^*)), \sigma \otimes \sigma', \delta + A \otimes 1 + 1 \otimes A' \right).
\end{align*}
One observes that the eigenvalues of the matrix $A \otimes 1 + 1 \otimes A'$ lie possibly outside the transversal $\trans$. However, there is a unique matrix $\tilde A$ with all its eigenvalues in $\trans$ such that
\begin{align}
\label{Atensor}
e^{2 \pi i \tilde A/\tau } := e^{2 \pi i (A \otimes 1 + 1 \otimes A')/\tau} = e^{2 \pi i A/\tau } \otimes e^{2 \pi i A'/\tau } \equiv \rho_1(1) \otimes \rho_1'(1).
\end{align}
The procedure of associating to $A \otimes 1 + 1 \otimes A'$ the matrix
$\tilde A$ defines the required map $c_{V,V'}$ since $\tilde A$ is the
connection matrix that one would have obtained (via $\F_\trans$) from $\rho_1\otimes
\rho_1'$. In fact, it follows that if $A \otimes 1 + 1 \otimes A'$ commutes
with $B \otimes B' \equiv \rho_2(1) \otimes \rho_2'(1)$ then so does $\tilde
A$. This map is natural in $V$ and $V'$ and the usual diagrams expressing
associativity and commutativity ({\it cf.} for instance \cite[Definition 1.8]{DelMil}) are satisfied. Moreover, it is bijective since an inverse can be constructed from Eqn. \eqref{Atensor} by using  the identification $\End_\CC(V \otimes V') = \End_\CC(V) \otimes \End_\CC(V')$ to obtain $A$ and $A'$ back from $\tilde A$.
\end{proof}
Note that the choice of the transversal $\trans$ is irrelevant since two
functors $\F_\trans$ and $\F_{\trans'}$ associated to two different
transversals $\trans$ and $\trans'$ to $\tau \ZZ$ are related via a natural
transformation that is given explicitly by a shearing transformation as
discussed before Proposition \ref{ConnMatrix}. 

We observe that it is also possible to prove the above equivalence directly by means of
the fibre functor $\omega$. For this we consider the full subcategory of
$\BB_q^\tau$ such that the connection matrices have all eigenvalues in the
same transversal $\trans$. It follows from Proposition \ref{ConnMatrix} that
this category is equivalent to $\BB_q^\tau$. By constructing the maps $c_{M,M'}$ very similar to those appearing in the above proof, one can show that this is an equivalence of rigid tensor categories.
Moreover, the restriction of the fibre functor gives it the structure of a
Tannakian category. The fibre functor induces an equivalence with
$\Rep(\ZZ^2)$ by defining the action of $\ZZ^2$ on $(\ker \nabla)_z$ to be
given by the matrices $e^{2 \pi i A /\tau}$ and $B$. Clearly, the functor $\F_\trans$ from the proof of Theorem \ref{thm:RH} is the inverse to this fibre functor.

\begin{rem}
For any group $H$ the category of its finite dimensional representations over
$\CC$ forms a neutral Tannakian category, which should be equivalent to the
category of representations of some affine group scheme, say $\hat{H}$. The group
scheme $\hat{H}$ is called the {\it algebraic hull} of $H$. Strictly speaking,
the affine group scheme underlying $\BB_q^\tau$ is the algebraic hull of
$\ZZ^2$. We refer the readers to Proposition 10.1 of \cite{Put} for an explicit computation of the
algebraic hull of $\ZZ$.

\end{rem}

\noindent
As a consequence we are able to conclude that the K-theory of $\BB_q^\tau$
 is the same as that of $\Rep(\ZZ^2)$. An object
of $\Rep(\ZZ^2)$ is a vector space $V$ equipped with two commuting linear
invertible endomorphisms. Using the fact that the two endomorphisms commute,
{\it i.e.,} respect each others eigenspaces,  one can always find a common
eigenvector $w$. This gives an exact sequence $0\map \langle w\rangle \map V
\map V/\langle w\rangle\map 0$ in $\Rep(\ZZ^2)$. Therefore, the K-theory of $\Rep(\ZZ^2)$ is the free abelian group generated by the simple objects, which are one dimensional
representations with two actions $a$ and $b$, with $a,b\in\CC^*$ (the actions
are given by multiplication by $a$ and $b$ respectively). The fibre functor sends the isomorphism class of $\left(\sheaf(\CC^*),b\alpha,\delta + z'\right)$ with $z' \in \CC$ to the simple object $(\CC,b,e^{2\pi i z'/\tau})$ in $\Rep(\ZZ^2)$. Note that
$\left(\sheaf(\CC^*),b\alpha,\delta + z'\right)$ and
$\left(\sheaf(\CC^*),b\alpha,\delta + (z'+n\tau)\right)$ are isomorphic via
the shearing transformation by $z^{n}$. Indeed,

\beqn
(\delta + z')z^nf = n\tau z^nf + z^n\delta f + z'z^nf = z^n\left(\delta
+(z'+n\tau)\right) f
\eeqn and their images also get identified via the
exponentiation. Summarising, we obtain

\begin{cor} \label{ktheory}
The K-theory of $\BB_q^\tau$ is the free abelian group generated by the isomorphism classes of the objects
$\left(\sheaf(\CC^*),b,\delta + z'\right)$ with $b\in \CC^*$ and $z'\in
\CC/\tau\ZZ$. Under this identification, one finds that the map on K-theory induced by the functor $\mathcal{S}_\tau\circ \psi_*$ sends
the class of $(\sheaf(\CC^*),b,\delta + z')$ to the divisor class of the point $-z'\in
X_\tau$ and their linear combinations accordingly. 
\end{cor}

\section{\'Etale fundamental group of $\TT_\theta^\tau$} 
In the context of noncommutative algebraic geometry typically one treats a small triangulated category with some geometric conditions (like finite total homological dimension) as a noncommutative space (see {\it e.g.,} \cite{KonNotes}). One views a classical smooth and proper algebraic variety in this setting via its bounded derived category of coherent sheaves. Roughly, Nori's notion of an \'etale fundamental group of a smooth and proper variety relies on realizing a Tannakian category structure on the subcategory of {\it Nori finite} bundles of the category of degree $0$ semistable bundles \cite{Nor1,Nor2}. We intend to formulate a similar subcategory in $\Vec(\TT_\theta^\tau)$ as a proposal for an \'etale fundamental group of the noncommutative torus. Let us briefly recall the (classical) algebraic setup.

A bundle $\Sheaf{F}$ on a smooth and proper algebraic variety $X$ is called {\it finite} if its class $[\Sheaf{F}]\in \K(X)$, where $\K(X)$ is a the Grothendieck group of the additive monoid of vector bundles on $X$, is integral over $\ZZ$. The {\it Nori finite bundles} are subquotients of finite directs sums of finite bundles. However, in characteristic $0$ Nori finite bundles are the same as finite bundles. 

Bridgeland's notion of a {\it stability condition} on an abstract triangulated category provides an axiomatic characterization of the subcategory of semistable bundles, which are sliced by their {\it phases}, a notion closely related to the degree of a bundle (in fact analogous to its slope) \cite{Bri1}. A stability condition on an abstract triangulated category $\mathcal{C}$ is a family of full additive subcategories $\mathcal{P}(\theta)$ parametrized by $\theta\in\RR$ and group homomorphism from 
$\mathcal Z:\K_0(\mathcal{C})\map \CC$ satisfying a few conditions that can be found in {\it ibid.} It turns out that specifying a stability condition is equivalent to specifying a bounded t-structure and a group homomorphism from the Grothendieck group of its heart to the complex numbers, sending the effective cone of the Grothendieck group to the upper half plane and satisfying the so-called {\it Harder--Narasimhan property} (Proposition 5.3 {\it ibid.}). 


The objects of $\mathcal{P}(\theta)$ are called {\it semistable objects} of phase $\theta$. This definition was motivated by an attempt to understand the ``stringy K\"ahler moduli" space of a Calabi--Yau manifold and so it is expected that a triangulated category equipped with a stability condition should behave like a noncommutative K\"ahler manifold. As a by-product we get hold of the subcategory of semistable objects $\mathcal{P}_\mathcal{C}:=\cup_{\theta\in\RR} \mathcal{P}(\theta)$, which comes with its natural slicing given by $\theta\in\RR$. It follows that each subcategory $\mathcal{P}(\theta)$ is abelian (Lemma 5.2 {\it ibid.}).

\begin{ex} \label{StabCond}
Let us look at the bounded derived category of coherent sheaves on a smooth projective curve $X$ denoted by $D^b(X)$, which Bridgeland used as a instructive example. We choose the standard t-structure on $D^b(X)$ and for any $\Sheaf{F}\in D^b(X)$ denote by $[\Sheaf{F}]$ its class in $\K_0(D^b(X))$. For any $\Sheaf{F}\neq 0$ define $\mathcal{Z}([\Sheaf{F}]) = -\deg(\Sheaf{F}) + i\rk(\Sheaf{F})$. This defines a stability condition and writing $\mathcal{Z}([\Sheaf{F}])=re^{\pi i\theta}$, we find that $\theta = \frac{1}{\pi}\arctan(\rk(\Sheaf{F})/(-\deg(\Sheaf{F})))=-\frac{1}{\pi}\arctan(1/(\mu(\Sheaf{F}))$. The subcategory $\mathcal{P}(\theta)$ corresponds to the abelian subcategory consisting of semistable bundles of slope $-1/\tan(\pi\theta)$. It turns out that the set of stability conditions on $D^b(X)$ (with some additional finiteness assumptions which we ignore here) admits a natural action of $\widetilde{GL}_2^+(\RR)$, which is free and transitive. For the case of higher genus, we refer to \cite{Mac}.
\end{ex}

\subsection{Semistable holomorphic bundles of degree $0$} \label{semistable}

There is an intrinsic definition of semistable bundles over $\TT_\theta^\tau$ (exactly mimicking slope stability) with a notion of Harder--Narasimhan filtration \cite{Pol6}. As observed before, a stability condition on $D^b(X_\tau)$ can be constructed by specifying a bounded t-structure and group homomorphism from the K-theory of its heart to the complex numbers mapping the effective cone to the upper half plane and having Harder--Narasimhan property. Let us consider the t-structure of Equation \eqref{Poltstr}. It follows from the main result of \cite{Pol1} that the K-theory of the heart of this t-structure, which is equivalent to $\Vec(\TT_\theta^\tau)$, is generated by the classes of the standard bundles $[E_{m,n}^z]$ such that its rank $\rk(E_{m,n}^z)=m\theta + n$ is positive. In fact, similar to Example \ref{StabCond} we obtain a stability condition on $D^b(\Vec(\TT_\theta^\tau))\cong$ $D^b(X_\tau)$ given by $\mathcal{Z} ([E_{m,n}^z]) = -\deg(E_{m,n}^z) + i\rk(E_{m,n}^z) = -m+i(m\theta +n)$. The Harder--Narasimhan property follows from the existence of such a filtration for holomorphic bundles (see Proposition 2.8.3 of \cite{Pol6}). This specifies the subcategory of semistable bundles of degree $0$, {\it i.e.,} $m=0$, which in terms of phase (with respect to this stability condition) is given by $\mathcal{P}(1/2)$. Let us denote this category by $\ssz(0,\TT_\theta^\tau)$; it is abelian. Recall that the category $\FrVec(\TT_\theta^\tau)$ introduced before Lemma \ref{FrVect} was defined as the extension closed subcategory of $\Vec(\TT_\theta^\tau)$ generated by the standard modules $E_1^z$. It follows from Proposition 2.1 part (b) of \cite{PolSch} that the classes $[E_1^z]$ and $[E_1^{z'}]$ are equivalent in $\K_0(\Vec(\TT_\theta^\tau))$ if and only if $z$ and $z'$ denote the same point in $X_\tau = \CC/(\ZZ+\tau\ZZ)$. So the classes $[E_1^z]$ account for the $\Pic^0(X_\tau)\cong X_\tau$ part of $\K_0(\Vec(\TT_\theta^\tau))$ and their extensions account for the full Picard group $\Pic(X_\tau)$.

\begin{lem}
The category $\FrVec(\TT_\theta^\tau)$ is a full subcategory of $\ssz(0,\TT_\theta^\tau)$.
\end{lem}

\begin{proof}
Since $E_1^z$ are indecomposable objects, the Calabi--Yau property of $D^b(X_\tau)$ shows that they must be semistable with respect to any stability condition ({\it e.g.,} Theorem 9.1 of \cite{Bri1}). Moreover, being degree $0$ objects with respect to the above stability condition they belong to $\ssz(0,\TT_\theta^\tau)$. Consequently, the objects obtained by extensions from them are also in $\ssz(0,\TT_\theta^\tau)$ which proves our assertion.
\end{proof}

\begin{rem}
The action of $\widetilde{GL}_2^+(\RR)$ on the stability manifold of an elliptic curve is transitive and it simply relabels the phases of semistable objects without altering them. So given any other choice of a stability condition on $D^b(X_\tau)$ one may apply an element of $\widetilde{GL}_2^+(\RR)$ to bring it to the form of our chosen stability condition $\mathcal{Z}$. 
\end{rem}

\subsection{Nori finite bundles over $\TT_\theta^\tau$}

The aim of this subsection is to identify the Nori finite bundles over $\TT_\theta^\tau$ and construct a Tannakian category structure on them. 

Let us begin by motivating our choice of $\ZZ$ as a proposed topological fundamental group of $\TT_\theta$. It was shown in \cite{ManMar,CDS} that noncommutative tori appear at the boundary $\mathbb{P}^1(\RR)/SL(2,\ZZ)$ of the moduli space of elliptic curves. At the rational points, which form to a single orbit under $SL(2,\ZZ)$-action, one has a degenerate elliptic curve, which is topologically equivalent to $\CC^*$. It is natural to put the degenerate elliptic curves and the noncommutative tori homotopically on the same footing. Let us consider more closely the case of rational $\theta$, say $\theta=p/q$. Then, the action of $\ZZ$ factors through the finite group $\ZZ_q$ and the crossed product $C^\infty(\SS^1) \rtimes_{p/q} \ZZ_q$ is Morita equivalent to $C^\infty(\SS^1/\ZZ_q) \simeq C^\infty(S^1)$ \cite{Rie74}. We consider the degenerate case with $\theta$ an irrational real number as homotopically equivalent to the finite group case.  
Therefore, the topological fundamental group of $\TT_\theta$ is expected to be $\ZZ$, and we thus look for a Tannakian subcategory of $\Vec(\TT_\theta^\tau)$ that is equivalent to $\Rep(\ZZ)$.

\begin{thm}
The category $\BB^\tau$ is a full subcategory of $\Vec(\TT_\theta^\tau)$.
\end{thm}
\begin{proof}
Since by Proposition \ref{modularity} $\Vec(\TT_\theta^\tau)$ and $\Vec(\TT_\theta^{g\tau})$ are equivalent for any $g \in SL(2,\ZZ)$, we can assume by Lemma \ref{lem:tau-width} that $\tau$ satisfies $\rwd(\tau)<1$. As observed in Remark \ref{functorBtau}, Proposition \ref{ourfunctor} actually defines a functor from $\BB^\tau$ to $\Vec(\TT_\theta^\tau)$. We also denote it by $\psi_*$. 
For $\BB^\tau$ there is an analogue of Proposition 18: for each object $(M,\nabla)$ in $\BB^\tau$ there is an isomorphic object $(V \otimes \sheaf(\CC^*), \delta + A)$ with $V$ a vector space and $A$ a constant matrix with all eigenvalues in the same transversal to $\tau \ZZ$. In fact, $\BB^\tau$ can be identified with the full subcategory of $\BB^\tau_q$ consisting of those objects with trivial $\ZZ$-action so that the argument in the proof of Prop. 18 applies directly. The proof of Theorem 19 then further implies that $\Hom_{\BB^\tau} ( (V \otimes \sheaf(\CC^*), \delta +A) , (V' \otimes \sheaf(\CC^*), \delta +A') )$ can be identified with the space of linear maps $\phi : V \to V'$ that intertwine $A$ and $A'$, {i.e.} so that $\phi A = A' \phi$. 

Next, consider morphisms in the image (under $\psi_*$) of $\BB^\tau$ inside $\Vec(\TT_\theta^\tau)$. The image of the object $(V\otimes \sheaf(\CC^*), \delta + A)$ is clearly $(V \otimes \A_\theta, \delta_\tau + 2\pi i A)$ (note the factor of $2 \pi i$). A morphism between the objects $(V\otimes  \A_\theta, \delta_\tau + 2\pi i A)$ and $(V'\otimes  \A_\theta, \delta_\tau + 2\pi i A')$ is then given as a map $\phi: V 
\to V' \otimes A_\theta$ that satisfies:
$$
\phi \circ ( \delta_\tau + 2 \pi i A) = ( \delta_\tau + 2 \pi i A')\circ \phi.
$$
We can decompose $\phi$ as follows: $\phi= \sum_{n_1,n_2 \in \ZZ} \phi_{n_1 n_2} U_1^{n_1} U_2^{n_2}$ and comparing powers of $U_i$ on both sides of the previous equation yields:
$$
 \phi_{n_1 n_2} A = ( ( \tau n_1 + n_2) + A') \phi_{n_1n_2}
$$
where we factored out a factor of $2 \pi i$ (cf. the definition of $\delta_\tau$). 

As before, this equation implies that $( \tau n_1 + n_2) + A'$ and $A$ have at least one eigenvalue in common. However, this is impossible (unless $n_1=n_2=0$) since the eigenvalues of $A$ and $A'$ lie in the same transversal: adding $\tau n_1 + n_2$ will translate every eigenvalue out of the transversal by our assumption $\rwd(\tau)<1$. Thus, $\phi_{n_1n_2}=0$ unless $n_1=n_2=0$, so that also a morphism in the image of $\BB^\tau$ in $\Vec(\TT_\theta^\tau)$ is given by a constant matrix $\phi \in \Hom_\C(V,V')$ satisfying $\phi  A = A' \phi$. 
\end{proof}
Actually, this shows that $\psi_*: \BB^\tau \to \Vec(\TT_\theta^\tau)$ is a full embedding. Moreover, the choice of a transversal $\trans$ to $\tau\ZZ$ yields a functor from $\Rep(\ZZ)$ to $\BB^\tau$, which -- by the Riemann--Hilbert correspondence -- is an equivalence of tensor categories. We denote this functor again by $F_\trans$ and via the composed functor $\psi_*\circ\F_\trans:\Rep(\ZZ)\to\Vec(\TT_\theta^\tau)$ we obtain some holomorphic bundles over $\TT_\theta^\tau$. These bundles are of degree 0, since the image of $\phi_* \circ F_\trans$ is 
contained in $\FrVec(\TT_\theta^\tau)$ which is a subcategory of $\ssz(0,\TT_\theta^\tau)$. 
This is our sought-after Tannakian subcategory in $\Vec(\TT_\theta^\tau)$, which is equivalent to $\Rep(\ZZ)$. 

This leads us to a proposal for the \'etale fundamental group of the noncommutative torus as follows. Recall that a holomorphic bundle $E$ over a smooth projective variety is called {\it finite} if there are two distinct polynomials $f$ and $g$ with non-negative integral coefficients such that $f(V)$ and $g(V)$ are isomorphic. In characteristic zero the Nori finite bundles are equivalent to finite bundles in the above sense. Over a smooth quasiprojective variety a holomorphic vector bundle obtained from a representation of its topological fundamental group that factors through a finite group is Nori finite \cite{Nor2}. The converse is also true, with some additional assumptions, over smooth quasiprojective varieties (see Theorem 3.3 \cite{BHS} for the precise statement). For a more elaborate treatment of Nori finite bundles over quasiprojective varieties in characteristic zero we refer the readers to \cite{EsnHai}. Motivated by this result we call the objects in the image of $\psi_*\circ\F_\trans$ which factor through a finite group representation {\it Nori finite} over $\TT_\theta^\tau$ and denote the full subcategory of $\Vec(\TT_\theta^\tau)$ that they comprise by $\NF(\TT_\theta^\tau)$. As a result of a Theorem proved by Nori in \cite{Nor2}, the category $\NF(\TT_\theta^\tau)$ is Tannakian.

\begin{rem}
That $\NF(\TT_\theta^\tau)$ is Tannakian can also be seen directly. Let $f$ be a morphism in $\NF(\TT_\theta^\tau)$ and up to isomorphism we may assume that it is of the form $f:(\A_\theta^{n_1},\delta_\tau + A_1)\to(\A_\theta^{n_2},\delta_\tau + A_2)$, where $A_i,$ $i=1,2$ are complex upper triangular matrices of appropriate sizes. The fact that they are in $\NF(\TT_\theta^\tau)$ means that all the entries are a rational numbers with all eigenvalues in the same transversal. It can be checked that the kernel and the cokernel of $f$ defined in $\Vec(\TT_\theta^\tau)$ will also have connection matrices with rational entries, which can be brought to the desired form up to isomorphism by suitable base changes. The rigid tensor structure on $\BB^\tau$ induces a rigid tensor structure on $\NF(\TT_\theta^\tau)$ with identity object $(\A_\theta,\delta_\tau)$. Moreover, that $\End(\A_\theta,\delta_\tau)=\CC$ can be seen from the fact that its image in $D^b(X_\tau)$ under the Polishchuk--Schwarz equivalence is the stable object $\sheaf_{z=0}$. Finally, the fibre functor on $\BB^\tau$ restricts to $\NF(\TT_\theta^\tau)$ turning the latter into a Tannakian category. 

\end{rem}

Since $\NF(\TT_\theta^\tau)$ is a subcategory of $\FrVec(\TT_\theta^\tau)$ its image under the functor $\mathcal{S}_\tau\circ\psi_*$ lies inside the category of torsion sheaves over $X_\tau$. On the other hand the category of Nori finite bundles of the complex elliptic curve $X_\tau$, denoted by $\NF(X_\tau)$, is also a subcategory of its category of semistable bundles of degree zero. For any $\mu\in\QQ$ it is known that the category of semistable bundles of slope $\mu$ over $X_\tau$ is equivalent to the category of torsion sheaves over $X_\tau$ by some Fourier--Mukai transform. In this specific case the category of semistable bundles of degree $0$ (hence slope $0$) can be made equivalent to the category of torsion sheaves by applying a specific Seidel--Thomas twist functor (for the details see \cite{SeiTho}). The effect of this functor on $(\deg\,\,\,\rk)^t$, viewed as a column vector, is simply left multiplication by the matrix $\left(\begin{smallmatrix}
0 & 1\\
-1 & 0 \\
\end{smallmatrix}\right)$, {\it i.e,} degree $0$ bundles become rank $0$ torsion sheaves. Therefore, $\NF(\TT_\theta^\tau)$ and $\NF(X_\tau)$ can both be viewed as subcategories of the category of torsion sheaves of $X_\tau$. The complex points of the Tannakian group of $\NF(X_\tau)$ is a isomorphic to $\hat{\ZZ}\times\hat{\ZZ}$, while that of $\NF(\TT_\theta^\tau)$ is simply $\hat{\ZZ}$. It is not very surprising since over an algebraically closed field of characteristic zero, Nori's fundamental group scheme is the same as Grothendieck's \'etale fundamental group (ignoring the difference that the former is a pro-group scheme and the latter is only an abstract group). The pinching of one of the homology cycles at the rational boundary of the modular curve (and possibly its deformation for noncommutative tori) accounts for the lack of one copy of $\hat{\ZZ}$ in the \'etale fundamental group of $\TT_\theta^\tau$. It conforms with the fact that the noncommutative torus is not homotopy equivalent to the commutative one, as seen by the difference in their ordered $\K_0$-groups.





\end{document}